\documentclass{article}
\usepackage{amsthm,amssymb,amsmath}
\usepackage{epsfig}
\usepackage{color}

\newtheorem{prelem}{{\bf Theorem}}

 \newtheorem{theorem}{Theorem}
\newtheorem{corollary}[theorem]{Corollary}
\newtheorem{lemma}[theorem]{Lemma}

\theoremstyle{definition}
\newtheorem{definition}[theorem]{Definition}
\theoremstyle{remark}

\title{Stanley depth of powers of the path ideal}
\author{Alin \c{S}tefan\\ Petroleum and Gas University of Ploie\c{s}ti\\ 
Ploie\c{s}ti, Romania\\
{\tt
nastefan@upg-ploiesti.ro}\vspace{3mm} \\
}

\date{}

\begin{document}
\maketitle

\begin{abstract}
The aim of this paper is to give a formula for the Stanley depth of quotient of powers of the path ideal.
As a consequence, we establish that the behaivior of the Stanley depth of quotient of powers of the path 
ideal is the same as  a classical result of Brodmann on depth.
\vspace{3mm}

\end{abstract}

\maketitle

\section{Introduction}
Let $S=K[x_{1},\ldots,x_{n}]$ be the polynomial ring in $n$ variables over a field 
$K$ and $M$ be a finitely generated $\mathbb{Z}^{n}$-graded $S$-module. 
Let $u\in M$ be homogeneous and $Z\subset X=\{x_{1},\ldots,x_{n}\}$.
Then the $K[Z]$-submodule $uK[Z]$ of $M$ is called a {\em Stanley space} 
of $M$ if $uK[Z]$ is a free $K[Z]$-submodule of $M$ and $|Z|$ is called the 
{\em dimension} of $uK[Z]$.
A {\em Stanley decomposition} $\mathcal{D}$ of $M$ is a decomposition of $M$ as a 
direct sum of $\mathbb{Z}^{n}$-graded $K$-vector space 
\[\mathcal{D} : M=\bigoplus_{j=1}^{r}u_{j}K[Z_{j}],\] where each $u_{j}K[Z_{j}]$ 
is a Stanley space of $M.$\\ 
The number \[sdepth(\mathcal(D)=min\{|Z_{i}| :  i=1,\ldots, r\})\]  is  called
the {\em Stanley depth of decomposition $\mathcal{D}$} and the number 
\[sdepth(M):=max\{sdepth(\mathcal{D}) : \ \mathcal{D} \ is \ a  \ Stanley \ 
decomposition \ of \ M\}\]
is called {\em Stanley depth} of $M.$\\ In 1982 Stanley conjectureted in $\cite{RS}$
that $sdepth(M)\geq depth(M).$ Apel $\cite{A1}, \cite{A2}$ proved the conjecture
for a monomial ideal $I$ over $S$ and for the quotient $S/I$ in 
at most three variables. Anwar and Popescu $\cite{AP}$ and Popescu$\cite{P}$ proved 
the conjecture  for $S/I$ and $n=4, 5$; also for $n=5$ Popescu proved the conjecture 
for square free monomial ideal. In $\cite{HVZ}$ Herzog, Vl\u{a}doiu and Zheng 
introduced a method to compute the Stanley depth of a factor of a monomial ideal which 
was later developed into an effective algorithm by Rinaldo \cite{Ri} implemented in 
{\em CoCoA} \cite{Co}. 
Also, the explicit computation of the Stanley depth turns out to be a dificult problem
even for simpler monomial ideals or quotient of monomial ideals. For instance in 
$\cite{BHKTY}$ Bir$\acute{o}$ et al. proved that $sdepth(m)=\left\lceil \frac{n}{2}
\right\rceil$ where $m=(x_{1},\ldots,x_{n})$ is the graded maximal ideal of $S$ and 
where $\left\lceil \frac{n}{2}\right\rceil$ denote the smallest integer 
$\geq \frac{n}{2}.$

\section{Stanley depth of path ideal}
Let $G$ be a graph on $n$ vertices. The {\em edge ideal} $I=I(G)$ of the graph $G$ 
is the ideal generated by all monomials of the form $x_{i}x_{j}$ such that 
$\{x_{i}, x_{j}\}$  is an edge of $G$.
\begin{definition}
A {\em path} $P_{n}$ of lenght $n-1, \ n\geq 2$ is a set of $n$ distinct vertices 
$x_{1},\ldots,x_{n}$ and $n-1$ edges $x_{i}x_{i+1}$ for $1\leq i\leq n-1.$
\end{definition}
For $I=I(P_{n}),$ Morey $\cite{M}$ proved  that $depth(S/I)=\left\lceil \frac{n}{3} 
\right\rceil$ and for the powers of $I$ is given a lower bound, 
$depth(S/I^{t})\geq max \{\left\lceil \frac{n-t+1}{3}\right\rceil, 1\}.$ The proof 
makes repeated use of applying the {\em Depth Lemma:}
\begin{lemma}$(\cite{V}$, $Lemma \ 1.3.9)$
If \[0\longrightarrow U \longrightarrow M \longrightarrow N \longrightarrow 0\] 
is a short exact sequence of modules over a local ring $R,$ then \\
$a)$ If $depth \ M < depth \ N,$ then $depth  \ U=depth \ M.$\\
$b)$ If $depth \ M > depth \ N,$ then $depth  \ U=depth \ N + 1.$
\end{lemma}
The most of the statments of the {\em Depth Lemma} are wrong if we replace {\em depth}
by  {\em sdepth.} Rauf $\cite{R}$ proved the analog of $Lemma \ 2(a)$ for {\em sdepth:}
\begin{lemma}
Let \[0\longrightarrow U \longrightarrow M \longrightarrow N \longrightarrow 0\] 
be an exact sequence of finitely generated $\mathbb{Z}^{n}$-graded $S$-modules.
Then \[sdepth \ M\geq min\{sdepth \ U, sdepth \ N\}.\]
\end{lemma}
\begin{lemma}
If $I=I(P_{n})$, then $sdepth(S/I)=\left\lceil \frac{n}{3}\right\rceil.$
\end{lemma}
\begin{proof}
For $n\leq 5$ the result holds very easy. The proof is by induction on $n\geq 6.$ 
Consider the short exact sequence:
\[0\longrightarrow S/(I:x_{n-1}) \stackrel{x_{n-1}}{\longrightarrow} S/I 
\longrightarrow S/(I,x_{n-1}) \longrightarrow 0.\] 
Firstly, note that $(I:x_{n-1})=(J,x_{n-2},x_{n})$ and $(I,x_{n-1})=(L,x_{n-2})$
 where $J=I(P_{n-3}), \ L=I(P_{n-2}).$ By induction on $n$ and 
($\cite{HVZ}$, $Lemma \ 3.6$), 
\[sdepth(S/(I:x_{n-1}))=sdepth(S^{'}[x_{n-1}]/J)=sdepth(S^{'}/J)+1=
\left\lceil \frac{n-3}{3} \right\rceil+1=\left\lceil \frac{n}{3}\right\rceil\]
where $S^{'}=K[x_{1},\ldots, x_{n-3}]$ and
\[sdepth(S/(I,x_{n-1}))=sdepth(K[x_{1},\ldots,x_{n-2}]/L)+1=\left\lceil \frac{n-2}{3} 
\right\rceil +1.\]
So, by $Lemma \ 3,$ $sdepth(S/I)\geq min\{sdepth(S/(I:x_{n-1})),sdepth(S/(I,x_{n-1}))\}=
\left\lceil \frac{n}{3}\right\rceil.$\\

Now we'll prove the another inequality, $sdepth(S/I)\leq \left\lceil \frac{n}{3}
\right\rceil.$ We identify $S/I$ with the $\mathbb{Z}^n$-graded $K$-subvector space 
$I^{c}$ of $S$ which is generated by all monomials $u\in S \setminus I.$\\
The {\em characteristic poset} (see $\cite{HVZ}$) of $S/I$ is 
 \[\mathcal{P}=\{a\in \mathbb{N}^{n}\ : \ x^{a}\in I^{c} \ and \
 x^{a} | x_{1}x_{2}\cdots x_{n}\},\] where $x^{a}=x_{1}^{a(1)}x_{2}^{a(2)}\cdots 
x_{n}^{a(n)}$ and $a=(a(1),\ldots,a(n))\in \mathbb{N}^{n}.$\\
For $d\in \mathbb{N}$ and $\alpha \in \mathbb{N}^{n}$ let
\[\mathcal{P}_{d}:=\{a\in \mathcal{P} \ :  |a|=d\} \ and \ \mathcal{P}_{d, \alpha}:=
\{a\in \mathcal{P}_{d} \ :  x^{\alpha} | x^{a}\},\]
where for $a=(a(1),\ldots,a(n))\in \mathbb{N}^{n},$ $|a|:=\sum_{i=1}^{n}a(i).$\\
We define a natural partial order on $\mathbb{N}^{n}$ as follows: $a\geq b$ if and only if 
$a(i)\geq b(i)$ for $i=1,\ldots, n$ and we say that $b$ $\it{cover}$ $a.$\\
Assume $sdepth(S/I)\geq \left\lceil \frac{n}{3}
\right\rceil +1$ $\Leftrightarrow$ there exists a partition of $\mathcal{P}=\bigcup_{i=1}^
{r}[F_{i},G_{i}]$ such that $min_{i=1}^{r}|G_{i}|=\left\lceil \frac{n}{3}\right\rceil+1.$
We denote by $\{e_{1},\ldots,e_{n}\}$ the canonical base of $\mathbb{R}^n$ and 
$\delta_{ij}$ the Kronecker'symbol.\\ We have three cases to study:
\begin{enumerate}
	\item If $n=3k\geq 6,$ then $\mathcal{P}_{k+1, \ \alpha}=\{\alpha +e_{3k-2}+e_{3k}\}$, 
	where $\alpha=\sum_{i=1}^{k-1}e_{3i-1}.$ 
	So, $A=\sum_{i=1}^{k}e_{3i-1}\in \mathcal{P}_{k}$ 
	is not covered. A contradiction!
	\item If $n=3k+1\geq 7,$ then $\mathcal{P}_{k+2, \ \alpha}=
	\{\alpha +e_{3k-5}+e_{3k-3}+e_{3k-1}+e_{3k+1}\}$, 
	where $\alpha=\sum_{i=1}^{k-2}e_{3i-1}.$ 
	So, $A=\sum_{i=1}^{k-1}e_{3i-1}
	\in \mathcal{P}_{k-1}$ is not covered. A contradiction!
		\item If $n=3k+2\geq 8,$ then $\mathcal{P}_{k+2, \ \alpha}=
	\{\alpha +e_{3k-2}+e_{3k}+e_{3k+2}\}$, where $\alpha=\sum_{i=1}^{k-1}e_{3i-1}.$ 
	So, $A=\sum_{i=1}^{k}e_{3i-1} \in \mathcal{P}_{k}$ 
	is not covered. A contradiction!
\end{enumerate}
Therefore, the required conclusion follows.
\end{proof}
\begin{lemma}
Let $t\geq 1,$ $G$ be a graph, $u, v$ vertices of $G,$ $I=I(G)$ such that 
$sdepth(S/(I^{t}:uv))\geq s,$ $sdepth(S/(I^t, u))\geq s$ and $sdepth(S/((I^{t}:u), v))
\geq s$ for some $s\geq 0,$ then $sdepth(S/I^{t})\geq s.$
\end{lemma}
\begin{proof}
Applying Lemma 3 to the short exact sequence 
\[0\longrightarrow S/(I^{t}:uv) \stackrel{v}{\longrightarrow} S/(I^{t}:u) 
\longrightarrow S/((I^{t}: u), v) \longrightarrow 0\] 
yields $sdepth(S/(I^{t}, u))\geq s.$ Applying again Lemma 3 to the sequence 
\[0\longrightarrow S/(I^{t}:u) \stackrel{u}{\longrightarrow} S/I^{t} 
\longrightarrow S/(I^{t}, u) \longrightarrow 0\] 
we have $sdepth(S/I^{t})\geq s.$
\end{proof}
\begin{theorem}
For $n\geq 2$ and $t\geq 1$ the power of the path ideal $I=I(P_{n})$ has 
Stanley depth, 
$sdepth(S/I^{t})=max\{\left\lceil \frac{n-t+1}{3}\right\rceil, 1\}.$
\end{theorem}
\begin{proof}
Since $P_{n}$ is a bipartite graph it follows $depth(R/I^{t})\geq 1, 
(see \ \cite{M})$ and so,
$sdepth(R/I^{t})\geq 1, (see \ \cite{C})$ for all $t\geq 1.$\\ 
We'll show that $sdepth(S/I^{t})\geq \left\lceil \frac{n-t+1}{3}\right\rceil.$
The proof is by induction on $n\geq 2$ and $t\geq 1.$ The result holds for 
$n\leq 3$ and for all $t\geq 1;$ also by Lemma 4 the result holds for 
$n\geq 2$ and $t=1.$\\
 Assume $n\geq 2$ and $t\geq 2.$ Using Lemma 5, it is enough to show that:
\begin{enumerate}
\item $sdepth(S/(I^{t}:x_{n-1}x_{n}))\geq \left\lceil \frac{n-t+1}{3}\right\rceil;$
\item $sdepth(S/(I^{t},x_{n-1}))\geq \left\lceil \frac{n-t+1}{3}\right\rceil;$
\item $sdepth(S/((I^{t}:x_{n-1}),x_{n}))\geq \left\lceil \frac{n-t+1}{3}\right\rceil.$
\end{enumerate}
\begin{enumerate}
\item By induction on $t$ and since $(I^{t}:x_{n-1}x_{n})=I^{t-1}$ we have
\[sdepth(S/(I^{t}:x_{n-1}x_{n}))=sdepth(S/I^{t-1})\geq \left\lceil \frac{n-(t-1)+1}{3}
\right\rceil\geq\left\lceil \frac{n-t+1}{3}\right\rceil\]
\item Firstly, note that $(I^{t},x_{n-1})=(J^{t},x_{n-1}),$ where $J=I(P_{n-2}).$ \\
By induction on $n$ and $(\cite{HVZ},  Lemma \ 3.6)$ we have: 
\[sdepth(S/(I^{t},x_{n-1}))=sdepth(S^{'}[x_{n-1}, x_{n}]/(J^{t},x_{n-1}))=S^{'}/J^{t}+1
\geq \left\lceil \frac{n-2-t+1}{3}\right\rceil+1=\]
\[\left\lceil \frac{n-t+2}{3}\right\rceil\geq \left\lceil \frac{n-t+1}{3}\right\rceil,\]
where $S^{'}=K[x_{1},\ldots,x_{n-2}].$
\item Using $(\cite{HM}, Theorem \ 3.5)$ it follows $((I^{t}:x_{n-1}), x_{n})=
((L^{t}:x_{n-1}), x_{n}),$ where $L=I(P_{n-1}).$ Notice that $sdepth(S/((I^{t}:x_{n-1}), x_{n}))=
sdepth(S^{''}/(L^{t}:x_{n-1})),$ where $S^{''}=K[x_{1},\ldots,x_{n-1}].$\\
Consider the short exact sequence:
\[0\longrightarrow S^{''}/(L^{t}:x_{n-1}x_{n-2}) \stackrel{x_{n-2}}{\longrightarrow} 
S^{''}/(L^{t}:x_{n-1}) \longrightarrow S^{''}/((L^{t}:x_{n-1}),x_{n-2}) \longrightarrow 0.\] 
By induction on $n$ and since $(L^{t}:x_{n-1}x_{n})=L^{t-1}$ we have:
\[sdepth(S^{''}/(L^{t}:x_{n-1}x_{n-2}))=sdepth(S^{''}/L^{t-1})\geq 
\left\lceil \frac{n-1-(t-1)+1}{3}\right\rceil=\left\lceil \frac{n-t+1}{3}\right\rceil.\]
Also, we have $((L^{t}:x_{n-1}), x_{n-2})=((Q^{t}:x_{n-1}), x_{n-2})=(Q^{t}, x_{n-2}),$ 
where $Q=I(P_{n-3}).$\\
Let $S^{'''}=K[x_{1},\ldots,x_{n-3}].$ By induction on $n$ and $(\cite{HVZ}, Lemma \ 3.6),$
\[sdepth(S^{''}/(Q^{t}, x_{n-2}))=sdepth(S^{'''}[x_{n-1}]/ Q^{t})=sdepth(S^{'''}/ Q^{t})+1\geq\]
\[ \left\lceil \frac{n-3-t+1}{3}\right\rceil+1=\left\lceil \frac{n-t+1}{3}\right\rceil.\]
 Applying Lemma 3 to the sequence above, 
$sdepth(S^{''}/(L^{t}, x_{n-1}))\geq \left\lceil \frac{n-t+1}{3}\right\rceil$ and so, we obtain
$sdepth(S/((I^{t}:x_{n-1}),x_{n}))\geq \left\lceil \frac{n-t+1}{3}\right\rceil.$
\end{enumerate}
Therefore, we obtain $sdepth(S/I^{t})\geq max\{\left\lceil \frac{n-t+1}{3}\right\rceil, 1\}$ 
for any $t\geq 1.$ 

Now we'll prove the another inequality, $sdepth(S/I^{t})\leq max\{\left\lceil \frac{n-t+1}{3}
\right\rceil, 1\}$ for any $t\geq 1.$ By Lemma 4 the result holds for $t=1.$
Let  $t\geq 2$ fixed.
We identify $S/I^{t}$ with the $\mathbb{Z}^n$-graded $K$-subvector space 
$(I^{t})^{c}$ of $S$ which is generated by all monomials $u\in S \setminus I^{t}.$\\
The {\em characteristic poset} (see $\cite{HVZ}$) of $S/I^{t}$ is 
 \[\mathcal{P}=\{a\in \mathbb{N}^{n}\ : \ x^{a}\in (I^{t})^{c} \ and \
 x^{a} | (x_{1}x_{2}\cdots x_{n})^{t}\},\] where $x^{a}=x_{1}^{a(1)}x_{2}^{a(2)}\cdots 
x_{n}^{a(n)}$ and $a=(a(1),\ldots,a(n))\in \mathbb{N}^{n}.$\\

Let us first show why $sdepth(S/I^{t})\leq 1$ for any $t\geq n-2.$
Assume $sdepth(S/I^{t})\geq 2$ for any $t\geq n-2.$ According to Theorem 2.1.(\cite{HVZ}) 
there exists a partition of $\mathcal{P}=\bigcup_{i=1}^{r}[F_{i},G_{i}]$ such that 
$min_{i=1}^{r}\rho(G_{i})=2,$ where $\rho(G_{i})= |\{j : t=G_{i}(j)\}|$ is the cardinality 
of $\{j : t=G_{i}(j)\}.$\\
 For $t\geq n-2$ fixed, let the sets $[(t,t-1,t,0,\ldots)]:=
\{(t,\alpha_{2},t,\alpha_{4},\beta)\in \mathcal{P} \ | \ 0\leq \alpha_{2}+\alpha_{4}\leq t-1, 
 \beta \in \mathbb{N}^{n-4}, \ |\beta|=(t-1)(\left\lceil \frac{n}{2}\right\rceil-2)-\alpha_{2}-\alpha_{4}\}$,  
$[(t-1,t-1,t,0,\ldots)]:=\{(t-1,\alpha_{2},t,\alpha_{4},\beta)\in \mathcal{P} \ | \  
0\leq \alpha_{2}+\alpha_{4}\leq t-1, \beta \in \mathbb{N}^{n-4}, 
|\beta|=(t-1)(\left\lceil \frac{n}{2}\right\rceil-2)-\alpha_{2}-\alpha_{4}\}$ and
$[(t,t-1,t-1,0,\ldots)]:=\{(t,\alpha_{2},t-1,\alpha_{4},\beta)\in \mathcal{P} \ | \ 
0\leq \alpha_{2}+\alpha_{4}\leq t-1, 
 \beta \in \mathbb{N}^{n-4}, \ |\beta|=(t-1)(\left\lceil \frac{n}{2}\right\rceil-2)-\alpha_{2}-\alpha_{4}\}.$
Since the elements of $[(t,t-1,t,0,\ldots)]$ can only cover the elements of $[(t-1,t-1,t,0,\ldots)] \cup
[(t,t-1,t-1,0,\ldots)]$ and there is an one to one corespondence 
between the sets $[(t,t-1,t,0,\ldots)]$ and $[(t-1,t-1,t,0,\ldots)]$ and 
for any $\gamma\in [(t,t-1,t,0,\ldots)],$  $\delta\in [(t-1,t-1,t,0,\ldots)]$ and $\eta \in [(t,t-1,t-1,0,\ldots)]$
we have $|\gamma|-1=|\delta|=\eta, \ \rho(\gamma)=2, \ \rho(\delta)=\rho(\eta)=1$ then
there exists elements from  $[(t-1,t-1,t,0,\ldots)] \cup [(t,t-1,t-1,0,\ldots)]$ which can not be covered 
by elements of $[(t,t-1,t,0,\ldots)].$
Therefore $sdepth(S/I^{t})\leq 1$ for any $t\geq n-2$ and so 
$sdepth(S/I^{t})= 1$ for any $t\geq n-2.$

Using the same technique as above we show why $sdepth(S/I^{t})\leq
\left\lceil \frac{n-t+1}{3}\right\rceil$ for any $2\leq t\leq n-3.$
Let $2\leq t\leq n-3$ fixed and we'll denote by $a:=\left\lceil \frac{n-t+1}{3}\right\rceil.$
Assume $sdepth(S/I^{t})\geq a+1.$ According to Theorem 2.1.(\cite{HVZ}) 
there exists a partition of $\mathcal{P}=\bigcup_{i=1}^{r}[F_{i},G_{i}]$ such that 
$min_{i=1}^{r}\rho(G_{i})=a+1.$\\
Let the sets $[(t,t-1\underbrace{t,0,\dots,t,0}_{a- times},\ldots)]:=
\{(t,\alpha_{2},t,\alpha_{4},\ldots,t,\alpha_{2a+2},\beta)\in \mathcal{P}\ | \ 
\sum_{i=1}^{a+1}\alpha_{2i}\leq t-1, \beta\in \mathbb{N}^{n-2a-2}, \ |\beta|=
(t-1)(\left\lceil \frac{n}{2}\right\rceil-a)-\sum_{i=1}^{a+1}\alpha_{2i}\}$ ,
$[(t-1,t-1\underbrace{t,0,\dots,t,0}_{a- times},\ldots)]:=
\{(t-1,\alpha_{2},t,\alpha_{4},\ldots,t,\alpha_{2a+2},\beta)\in \mathcal{P}\ | \ 
\sum_{i=1}^{a+1}\alpha_{2i}\leq t-1, \beta\in \mathbb{N}^{n-2a-2}, \ |\beta|=
(t-1)(\left\lceil \frac{n}{2}\right\rceil-a)-\sum_{i=1}^{a+1}\alpha_{2i}\}$ and 
$[(t,t-1\underbrace{t,0,\dots,t,0}_{a-1- times},\ldots)]:=
\{(t,\alpha_{2},t,\alpha_{4},\ldots,t,\alpha_{2a},\beta)\in \mathcal{P}\ | \ 
\sum_{i=1}^{a}\alpha_{2i}\leq t-1, \beta\in \mathbb{N}^{n-2a}, \ |\beta|=
(t-1)(\left\lceil \frac{n}{2}\right\rceil-a+1)-\sum_{i=1}^{a}\alpha_{2i}\}.$
Since the elements of $[(t,t-1\underbrace{t,0,\dots,t,0}_{a- times},\ldots)]$ can
only cover the elements of $[(t-1,t-1\underbrace{t,0,\dots,t,0}_{a- times},\ldots)]$ or 
$[(t,t-1\underbrace{t,0,\dots,t,0}_{a-1- times},\ldots)]$ and
there is an one to one corespondence 
between the sets $[(t,t-1\underbrace{t,0,\dots,t,0}_{a- times},\ldots)]$ and 
$[(t-1,t-1\underbrace{t,0,\dots,t,0}_{a- times},\ldots)]$ and
for any $\gamma\in [(t,t-1\underbrace{t,0,\dots,t,0}_{a- times},\ldots)]$,  
$\delta\in [(t-1,t-1\underbrace{t,0,\dots,t,0}_{a- times},\ldots)]$ and 
$\eta\in [(t,t-1\underbrace{t,0,\dots,t,0}_{a-1- times},\ldots)]$
we have $|\gamma|-1=|\delta|=|\eta|, \ \rho(\gamma)=a+1, \ \rho(\delta)=a, \ \rho(\eta)=a$ 
then there exists elements from $[(t-1,t-1\underbrace{t,0,\dots,t,0}_{a- times},\ldots)]\cup 
[(t,t-1\underbrace{t,0,\dots,t,0}_{a-1- times},\ldots)]$ which can not be covered by elements
of $[(t,t-1\underbrace{t,0,\dots,t,0}_{a- times},\ldots)].$
Therefore $sdepth(S/I^{t})\leq
\left\lceil \frac{n-t+1}{3}\right\rceil$  and so we have the equality $sdepth(S/I^{t})=
\left\lceil \frac{n-t+1}{3}\right\rceil$  for any $2\leq t\leq n-3.$\\
Thus, we have $sdepth(S/I^{t})=max\{\left\lceil \frac{n-t+1}{3}\right\rceil, 1\}$ for any $t\geq 1.$
\end{proof}
By a theorem of Brodmann (\cite{B}), $depth(S/I^{t})$ is constant for $t>> 0.$ In (\cite{M}) 
Morey  proved  that $depth(S/I)=\left\lceil \frac{n}{3} 
\right\rceil$ and for the powers of $I$ is given a lower bound, 
$depth(S/I^{t})\geq max \{\left\lceil \frac{n-t+1}{3}\right\rceil, 1\}.$ As a consequence of 
the previous theorem we obtain a similar result to Brodmann' theorem on the Stanley depth.

\begin{corollary}
Stanley depth of factor of power of path ideal stabilizes, ie $sdepth(S/(I(P_{n}))^{t})=1$ 
for any $t\geq n-2.$
\end{corollary}

\end{document}